\documentclass[a4paper,12pt]{amsart}
\usepackage[utf8]{inputenc}
\usepackage{geometry}
\usepackage{amsmath}
\usepackage{amsfonts,amssymb}
\usepackage{appendix}
\usepackage{tikz}
\usepackage{epsfig}
\usepackage{amssymb}
\usepackage{mathrsfs}
\usepackage{amscd} 
\usepackage[all]{xy}
\usetikzlibrary{cd}

\newtheorem{theorem}{Theorem}[section]
\newtheorem{lemma}[theorem]{Lemma}

\theoremstyle{definition}
\newtheorem{definition}[theorem]{Definition}

\begin{document}
\title{Equivariant Poincar\'e-Hopf theorem}

\author[Hongzhi Liu]{Hongzhi Liu}
\address{\normalfont{School of Mathematics, Shanghai University of Finance and Economics}} \email{liu.hongzhi@sufe.edu.cn} 

\author[Hang Wang]{Hang Wang} 
\address{\normalfont{Research Center for Operator Algebras, School of Mathematical Sciences\\
East China Normal University, 3663 North Zhongshan Rd,\\
Shanghai 200062, P.R.China}}
\email{wanghang@math.ecnu.edu.cn}

\author[Zijing Wang]{Zijing Wang}
\address{\normalfont{Research Center for Operator Algebras, School of Mathematical Sciences\\
East China Normal University, 3663 North Zhongshan Rd,\\
Shanghai 200062, P.R.China}}
\email{52205500014@stu.ecnu.edu.cn}

\author[Shaocong Xiang]{Shaocong Xiang}
\address{\normalfont{Research Center for Operator Algebras, School of Mathematical Sciences\\
East China Normal University, 3663 North Zhongshan Rd,\\
Shanghai 200062, P.R.China}}
\email{52265500015@stu.ecnu.edu.cn}

\begin{abstract}
    In this paper, we employ the framework of localization algebras to compute the equivariant K-homology class of the Euler characteristic operator, a central object in studying equivariant index theory on manifolds. This approach provides a powerful algebraic language for analyzing differential operators on equivariant structures and allows for the application of Witten deformation techniques in a K-homological context. Utilizing these results, we establish an equivariant version of the Poincaré-Hopf theorem, extending classical topological insights to the equivariant case, inspired by the results of L\"uck-Rosenberg. This work thus offers a new perspective on the localization techniques in the equivariant K-homology, highlighting their utility in deriving explicit formulas for index-theoretic invariants.
\end{abstract}

	\maketitle

	\section{Introduction}
	The Witten deformation, a powerful technique introduced by Witten in 1982, involves a one-parameter family of Laplace-like operators acting on the algebra of differential forms on manifolds obtained by deforming the de Rham operator \cite{ssm}. Witten demonstrated that, by examining the asymptotic behavior of the spectrum of these operators, the Morse inequalities for the deforming function could be recovered analytically.
    
    This analytical approach to problems coming from topology has been adapted to the realm of index theory for elliptic operators, providing a powerful asset for the derivation of explicit formulas in index theory. For instance, it enables a purely analytic proof of the Poincar\'e-Hopf index formula \cite{z}. Moreover, in \cite{rj}, following the idea of Witten deformation, 
    Rosenberg proved that the K-homology class of the Euler characteristic operator is as trivial as one can imagine, i.e. comes from the embedding of points. Later,  L\"uck and Rosenberg established an equivariant version of this theorem in \cite{lr}. 


   In L\"uck and Rosenberg's proof in \cite{lr}, the key idea is to use some deformed de Rham operator to simplify problems by localizing around the zero points of a given vector field. 
    
    To be more precise, let $\Gamma$ be a countable discrete group and $M$ be an even-dimensional cocompact proper smooth $\Gamma$-manifold without boundary, equipped with a $\Gamma$-invariant Riemannian metric. The Euler characteristic operator
	$$ d+d^* \colon \Omega^*(M)  \rightarrow \Omega ^*(M) , $$
	with the $\mathbb{Z}_2$-grading $\Omega^*(M)=\Omega^*(M)^{even} \oplus \Omega^*(M)^{odd}$ is an equivariant differential operator. Let $\Xi$ be an equivariant vector field on $M$ which is non-degenerate and transverse to the zero-section, then we define the operator
    $ \hat{c}(\Xi) :\Omega^{*}(M) \rightarrow \Omega^{*}(M) $ 
    given by the right Clifford multiplication by $\Xi$, twisted by the grading. Namely, for $ \omega\in \Omega^{*}(M)$, we have
    \begin{align*}
        \hat{c}(\Xi)(\omega)
        =
        \begin{cases}
            \omega \cdot \Xi, \hspace{3.5em} \omega \in  \Omega^{*}(M)^{even}, \\
            \omega \cdot (-\Xi), \hspace{2em} \omega \in  \Omega^{*}(M)^{odd}, 
        \end{cases}
    \end{align*}
    where $\cdot$ is the Clifford multiplication. The deformed Euler characteristic operator we consider is $d+d^*+\lambda \hat{c}(\Xi)$ for some $\lambda \in [0,\infty)$. In \cite{lr}, L\"uck and Rosenberg consider some Fredholm module determined by $d+d^*+\lambda \hat{c}(\Xi)$, which, as $\lambda \rightarrow \infty$, yields a new Fredholm module depending only on the inclusion $\mathrm{Zero}(\Xi) \hookrightarrow M$ and is still homotopic to the Fredholm module determined by $d+d^*+\lambda \hat{c}(\Xi)$, where $\mathrm{Zero}(\Xi)$ is the set of zero points of $\Xi$. Hence, the K-homology class of $d+d^*$ can be determined by examining the neighborhood of the zero set of $\Xi$.

    In this paper, we use the language of localization algebras (\cite{yu1997localization}) to generalize the idea of Witten deformation. We will also use this idea to calculate the class of the Euler characteristic operator in the K-theory of the equivariant localization algebra of the manifold, i.e. the equivariant K-homology class represented by the Euler characteristic operator. This approach provides a new proof of the main theorem in \cite{lr}. Our main result is the following.

    \begin{theorem}
    \label{main}
        The K-homology class $[d+d^*]$ is in the image of the following maps:
        $$
        \bigoplus_{j=1}^{m} K_{0}(C^*(\Gamma_{j}))
        \xrightarrow[\oplus_{j=1}^{m}\varphi_{j}]{\cong}
        \bigoplus_{j=1}^{m} K_{0}(RC_{L}^*( N_{\varepsilon}(x_{j}) )^{\Gamma_{j}})
        \xrightarrow[r_{*}]{\cong}
        K_{0}\left( RC_{L}^{*} (\bigcup_{j \in \mathbb{N}} N_{\varepsilon}(x_{j})) ^{\Gamma} \right)
        \xrightarrow[i_{*}]{}
        K_{0}^{\Gamma} (M). $$
        We have: 
        $$[d+d^*]=\sum_{j=1}^{m} \mathrm{ind} (\Xi,x_{j})\cdot i_{*} \circ r_{*} \circ \varphi_{j}  ([I_{j}]).$$
    \end{theorem}

    Throughout this paper, $ \{ x_{j} \}_{j \in \mathbb{N} } $ denotes the collection of non-degenerate zero points of $\Xi$, $\mathrm{ind} (\Xi,x_{j})$ is the index of $\Xi$ at its zero point $x_{j}$, defined as
    $$ \mathrm{ind} (\Xi,x_{j}) = \mathrm{sign} \left( \mathrm{det} \left( d_{x_{j}}\Xi : T_{x_{j}}M \rightarrow T_{x_{j}}M \right) \right) \in \{ \pm 1 \} , $$ 
    and the notation $ N_{\varepsilon} (x_{j}) $ refers to the tubular neighborhood of $x_{j}$ with a small width $ \varepsilon >0 $. Since the zero set $ \mathrm{Zero}(\Xi) $ is discrete — due to $\Xi$ being transverse to the zero-section with finite orbits under the cocompact group action — we designate $x_{1}, \cdots, x_{m}$ as the representatives of the orbits in $ \mathrm{Zero}(\Xi)$ for simplicity. Additionally, we choose $\varepsilon$ sufficiently small so that the neighborhoods $ N_{\varepsilon}(x_{j}) $, $j=1,\cdots, m$, do not intersect with each other. For each given $j$, denote the stabilizer group of $x_{j}$ by
    $$ \Gamma_{j} := \{ \gamma \in \Gamma : \gamma \cdot x_{j} = x_{j} \}, $$
    and the trivial representation of $\Gamma_{j}$ by $I_{j}$ . It is easy to check that $N_{\varepsilon}(x_{j})$ is $ \Gamma_{j} $-invariant. 
    
    The notation $ C_{L}^{*}(X) ^{G} $ refers to the localization algebra of a locally compact Hausdorff space $X$ under the action of a discrete group $G$ (see Definition \ref{def loc}), and $RC_{L}^{*}(X) ^{G}$ denotes the representable localization algebra of $X$ (see Definition \ref{def rloc}). 

    Notice that if $\Gamma$ is trivial, then each $\Gamma_{j}$ is also trivial so that the class of the trivial representation $[I_{j}]$ is identified with $1$ in $K_{0}(C^{*}(\Gamma_{j})) \cong \mathbb{Z}$.  In this case, Theorem~\ref{main} reduces to the original Poincar\'e-Hopf theorem.
    
    Our primary strategy for proving Theorem \ref{main} is to restrict the class of the Euler characteristic operator to the $\varepsilon$-neighborhood $\bigcup_{j \in \mathbb{N}} N_{\varepsilon}(x_{j})$ by analyzing the propagation of elements in the localization algebra. First, we shall prove the following theorem:

    \begin{theorem}
    \label{1.2}
        The K-homology class of $d+d^{*}$ lies in the image of the map
		$$ i_{*} :  
        K_{0}\left( RC_{L}^{*} (\bigcup_{j \in \mathbb{N}} N_{\varepsilon} ( x_{j} ))^{\Gamma} \right)
        \rightarrow
        K_{0}\left( C_{L}^{*}(M) ^{\Gamma} \right) ,  $$
		where $i_{*}$ denotes the map induced on representable K-homology induced by the inclusion of $ \bigcup_{j \in \mathbb{N}} N_{\varepsilon} ( x_{j} ) $ into $M$:
		$$ i \colon \bigcup_{j \in \mathbb{N}} N_{\varepsilon} ( x_{j} ) \hookrightarrow M. $$
	\end{theorem}
    
    After we localized the equivariant $K$-homology class of the operator, it is necessary to determine the extent of information obtainable from each inclusion $ N_{\varepsilon}(x_{j}) \hookrightarrow M$ for each $x_{j}$. By utilizing the balanced product construction, we can express $ \bigcup_{j \in \mathbb{N}} N_{\varepsilon} ( x_{j} ) $ as 
    $$ \bigcup_{j=1}^{m} \Gamma \times_{\Gamma_{j}} N_{\varepsilon} ( x_{j} ). $$
    Since $ N_{\varepsilon}(x_{j}) $, $j=1,\cdots$ are pairwise disjoint, we could view $ K_{0}\left( RC_{L}^{*} (\bigcup_{j \in \mathbb{N}} N_{\varepsilon} ( x_{j} ))^{\Gamma} \right) $ as $ \bigoplus_{i=1}^{m} K_{0}\left( RC_{L}^{*} ( \Gamma \times_{\Gamma_{j}} N_{\varepsilon} ( x_{j} ))^{\Gamma} \right).$ By induction,  we have the following isomorphism
    $$ K_{0}\left( RC_{L}^{*} (bigcup_{j \in \mathbb{N}} N_{\varepsilon} ( x_{j} ))^{\Gamma} \right) 
    \cong 
    \bigoplus_{i=1}^{m} K_{0}\left( RC_{L}^{*} ( \Gamma \times_{\Gamma_{j}} N_{\varepsilon} ( x_{j} ))^{\Gamma} \right) 
    \cong 
    \bigoplus_{i=1}^{m} K_{0}\left( RC_{L}^{*} ( N_{\varepsilon} ( x_{j} ))^{\Gamma_{j}} \right) . $$
    
    After that, we will construct explicit classes in $ K_{0}\left( RC_{L}^{*} ( N_{\varepsilon} ( x_{j} ))^{\Gamma_{j}} \right) $ for all $ j = 1 , \cdots , m $ to serve as a pre-image of $[d+d^{*}]$ under $i_{*}$. These classes are determined by $ t^{-1} (d+d^{*}) + \lambda \hat{c}(\Xi) $ and will be denoted by $ [d+d^{*}]|_{N_{\varepsilon}(x_{j})} $ (see Definition \ref{def}) for convenience. 
    Notice that $ N_{\varepsilon}(x_{j}) $ is diffeomorphic to $\mathbb{R}^n$, which facilitates the identification of the $\Gamma_{j}$-space $ N_{\varepsilon}(x_{j}) $ with $\mathbb{R}^n$ where the $ \Gamma_{j} $-action on $\mathbb{R}^n$ is induced by the diffeomorphism. We will then prove the following theorem:
    
	\begin{theorem}
    \label{1.3}
	    For each $ j = 1 , \cdots , m $, there exists an isomorphism
        $$ \mathrm{PD}_{j} : K_{0}^{\Gamma_{j}}(\mathbb{R}^n) \stackrel{\cong}{\longrightarrow} K_{0}\left( RC_{L}^{*} (N_{\varepsilon}(x_{j}))^{\Gamma_{j}} \right), $$
        which is given by Poincar\'e duality and maps the class of the Bott operator in $ K_{0}^{\Gamma_{j}}(\mathbb{R}^n) $ to $ \mathrm{ind} (\Xi,x_{j}) \cdot [d+d^{*}]|_{N_{\varepsilon}(x_{j})} $.
	\end{theorem}

     Recall the standard result of Bott periodicity in \cite{ph}, which states that there exists a canonical isomorphism between $ K_{0}^{\Gamma_{j}}(\mathbb{R}^n) $ and $ K_{0}(C^{*}(\Gamma_{j})) $, under which the class of the Bott operator in $ K_{0}^{\Gamma_{j}}(\mathbb{R}^n) $ corresponds to the class of the trivial representation $ [I_{j}] \in K_{0}(C^{*}(\Gamma_{j})) $. Thus, our main theorem follows directly as a corollary of these results.

    This paper is organized as follows:
    \begin{enumerate}
        \item Preliminaries: In this section, we will establish some basic properties of elements in localization algebras.
        \item Equivariant Euler operator and localization trick: In this section, we will demonstrate Theorems \ref{1.2} and \ref{1.3}.
        \item Higher index map and Poincaré-Hopf Theorem on K-homology level: In this section, we will prove Theorem~\ref{main}, which provides an equivariant version of the Poincaré-Hopf index formula.
    \end{enumerate}

	\section{Preliminaries}
	In this section, we review the construction of geometric $C^{*}$-algebras and the higher index theory. See \cite{yu} for further details.

	\subsection{Equivariant Roe algebras and Equivariant localized Roe algebras}
	
	Let $X$ be a locally compact space on which a countable discrete group $\Gamma$ acts properly. Then, the action of $\Gamma$ on $X$ induces an action of $\Gamma$ on $C_0(X)$, defined by $(\gamma \cdot f)(x) := f(\gamma^{-1}x)$ for every $\gamma \in \Gamma$, $f \in C_0(X)$, and $x \in X$.

    \begin{definition}
        An $X$-$\Gamma$-module is a separable Hilbert space $H_{X}$ equipped with 
        \begin{enumerate} 
            \item [(1)] a non-degenerate representation $ \rho \colon C_0(X) \rightarrow B(H_X) $ , where $ B(H_X)$ is the set of all bounded linear operators on $H_X$;
            \item [(2)] a unitary representation $ U \colon \Gamma \rightarrow U(H_X) $ , where $ U(H_X)$ is the set of all unitary operators on $H_X$;
            \item [(3)] the unitary representation $U$ spatially implements the action of $\Gamma$ on $C_0(X)$ , i.e. 
            $$ U(\gamma) f U(\gamma)^*= \gamma \cdot f  $$
            for any $\gamma\in \Gamma$ and $ f \in C_0(X) $.
        \end{enumerate}
        
        An $X$-$\Gamma$-module is called standard if no nonzero function in $C_{0}(X)$ acts as a compact operator. An $X$-$\Gamma$-module is called admissible if, for any finite subgroup $F$ of $\Gamma_{ j }$ and for any $F$-invariant Borel subset $E$ of $X$, there exists a Hilbert space $H_{E}$ equipped with the trivial representation of $F$ such that $\chi_E H_{E}$ and $l^{2}(F) \otimes H_{E}$ are isomorphic as representations of $F$.

        In this context, $\chi_{E}$ denotes the characteristic function of the set $E$, and its action on $H_{E}$ is defined by a representation mapping from the space of all Borel functions on $X$, $ \mathrm{Borel}(X) \rightarrow B(H_X)$, which is extended by the representation $\rho \colon C_0(X) \rightarrow B(H_X)$.

    \end{definition}

    For example, if $X$ is a connected Riemannian manifold with non-zero dimension, whether complete or not, then for any Hermitian vector bundle $E$ on $X$, $L^2(E)$ serves as a standard $X$-module. Moreover, if $X$ is equipped with a proper isometric $\Gamma$-action, then $L^2(E)$ is a standard admissible $X$-$\Gamma$-module. For more details, see chapter 4 of \cite{yu}.

	\begin{definition}\label{def loc}
		Let $H_X$ be a standard $X$-$\Gamma$-module, and let $ T \colon H_X \rightarrow H_X $ be a bounded linear operator.
		\begin{enumerate}
			\item [(1)] 
			The support of $T$ , denoted $\operatorname{supp}(T)$ , consists of all points $ (x,y) \in X\times X  $ such that for all open neighbourhoods $U$ of $x$ and $V$ of $y$ 
			$$ \rho(\chi_{U}) T \rho(\chi_{V}) \not= 0 . $$
			where $\chi_{U}$ and $\chi_{V}$ are the characteristic functions of $U$ and $V$ respectively. 
			\item [(2)]
			If $X$ is moreover a metric space with metric $d$, define the propagation speed of $T$, denoted $\mathrm{prop}(T)$, by the diameter of its support, i.e.
            $$ \mathrm{prop}(T) = \sup \{ d(x,y) : (x,y)\in \mathrm{supp}(T) \}. $$
            \item [(3)]
			$T$ is said to be locally compact if both $\rho(f)T$ and $T\rho(f)$ are compact for all $ f \in C_0(X) $.
		\end{enumerate}
	\end{definition}
	
	\begin{definition}
		Let $H_X$ be a standard $X$-$\Gamma$-module, the $\Gamma$-equivariant localization algebra of $X$, denoted  $C_L^*(X)^{\Gamma}$, is the $C^*$-algebra generated by the collection of all bounded functions $ T_{t} $ from $ [ 1, \infty ) $ to $ B(H_{X}) $ such that
		\begin{enumerate}
			\item [(1)]
			for any $\Gamma$-compact subset $K$ of $X$ and denoted $\chi_{K}$ as the characteristic function of $K$ , there exists $t_K\geq 1$ such that for all $ t\geq t_K $, the operators
			$$ \chi_{K} T_{t} \quad  \text{and} \quad  T_{t} \chi_{K} $$ 
			are compact, and the functions
			$$ t \mapsto \chi_{K} T_{t} \quad \text{and} \quad   t\mapsto T_{t} \chi_{K} $$ 
			are uniformly norm continuous when restricted to $  \left[ t_K , \infty \right) $;
			\item [(2)]
			for any open neighbourhood $U$ of the diagonal in $ X^{+} \times X^{+} $, where $X^{+}$ is the one point compactification of $X$ ,  there exists $t_U \geq 1$ such that for all $ t > t_U $ 
			$$ \operatorname{supp} (f(t)) \subseteq U . $$
			
		\end{enumerate}
	\end{definition}
	
	The localization algebra of $X$ is independent (up to isomorphisms) of the choice of standard $X$-$\Gamma$-modules.

    Then, we define the representable localization algebra.

    \begin{definition}\label{def rloc}
      The equivariant representable localization algebra of $X$ denoted  by $RC_L^*(X)^{\Gamma}$  is generated by all equivariant bounded and uniformly norm-continuous functions $ f \colon \left[ 1 , \infty \right) \rightarrow B(H_X) $ such that
	   \begin{enumerate}
			\item [(1)]
			$ f \in C_{L}^{*}(X)^{\Gamma} $ 
            \item [(2)] 
            for any $ \Gamma $-compact subset $K$ of $X$, there exists $t_K>0$ such that $\chi_{K}T_{t}\chi_{K}=T_{t}$ for any $t>t_K$.
		\end{enumerate}
    \end{definition}

    It is direct from the definition that for a cocompact space, its localization algebra and representable localization algebra are identical.

	\subsection{Local index}
	In this subsection, we recall the definition of the local index for first-order elliptic differential operators. Under the isomorphism between the K-theory group of localization algebras and the K-homology group defined via KK-theory, the local index of a first-order elliptic differential operator is identified with its K-homology class. For more details, see \cite{yu1997localization} and \cite{qr}.
	
	Let $X$ be a complete Riemannian manifold with a proper, cocompact, and isometric action of a discrete group $\Gamma$, and let $D$ be a first-order $\Gamma$-equivariant elliptic differential operator acting on a Hermitian vector bundle $E$ over $X$. Let $H$ denote the Hilbert space of $L^2$-sections of the bundle $E$, a standard, non-degenerate, and admissible $X$-$\Gamma$-module. We first assume that $\operatorname{dim} X$ is even.

    In constructing the local index of $D$, we introduce a notation called the normalizing function. In this paper, we choose a specific normalizing function $\chi \colon (-\infty,\infty) \rightarrow (-1,1)$ defined as
	$$ \chi(x) = \frac{2}{\pi} \int_{0}^{x} \frac{1-\cos y}{y^2} dy. $$ 
	
	The function $\chi$ satisfies the following properties: 
    \begin{enumerate}
        \item [(1)]
        $\chi$ is a non-decreasing odd function, and $\lim\limits_{x \to \pm \infty} \chi(x) = \pm 1$.
        \item [(2)]
        The distributional Fourier transform of $\chi$ is supported in $[-1,1]$.
        \item [(3)]
        $1 - \chi(x)^2 < \frac{8}{\pi |x|}$.
    \end{enumerate}

    Consider the functional calculus $\chi(t^{-1}D)$, where $t \in [1, \infty)$. This serves as a multiplier of the localization algebra $C_{L}^{}(X)^{\Gamma}$ and is unitary modulo $C_{L}^{}(X)^{\Gamma}$. This gives rise to an element in $K_{0}(C_{L}^{*}(X)^{\Gamma})$, which we refer to as the local index of $D$.

    Next, we focus on the Euler characteristic operator $d + d^{*}$ and its local index. As we need to consider the deformation of $d + d^{*}$, we will also examine a deformation of its local index. Regarding the $\mathbb{Z}_{2}$-grading of $d + d^{*}$, $\chi(t^{-1}(d + d^{*}) + \lambda \hat{c}(\Xi))$ is an odd self-adjoint operator for any positive $t$ and $\lambda$, and it has the following form:
	$$ \chi(t^{-1}(d+d^*)+\lambda \hat{c}(\Xi)) = \begin{pmatrix}
		0 & U_{t,\lambda} \\
		V_{t,\lambda} & 0
	\end{pmatrix} . $$
	
	Set
	$$ W_{t,\lambda} = 
	\begin{pmatrix}
		1 & U_{t,\lambda} \\
		0 & 1
	\end{pmatrix}
	\begin{pmatrix}
		1 & 0 \\
		-V_{t,\lambda} & 0
	\end{pmatrix}
	\begin{pmatrix}
		1 & U_{t,\lambda} \\
		0 & 1
	\end{pmatrix}
	\begin{pmatrix}
		0 & -1 \\
		1 & 0
	\end{pmatrix} , 
    q= 
    \begin{pmatrix}
		1 & 0 \\
		0 & 0
	\end{pmatrix}
	$$
	and 
\begin{align*}
     P_{t,\lambda} &= W _{t,\lambda} \, q \, W _{t,\lambda}^{-1} \\
     &=\begin{pmatrix}
	1-(1-U_{t,\lambda}V_{t,\lambda})^2 & U_{t,\lambda}(1-V_{t,\lambda}U_{t,\lambda}) + (1-U_{t,\lambda}V_{t,\lambda})U_{t,\lambda}(1-V_{t,\lambda}U_{t,\lambda})  \\
		V_{t,\lambda}(1-U_{t,\lambda}V_{t,\lambda}) & (1-V_{t,\lambda}U_{t,\lambda})^2
	\end{pmatrix}.
\end{align*}

	The local index of $d+d^*$ can be represented by
	$$ [t\mapsto P_{t,\lambda}] - [q] \in K_{0}\left( C_{L}^{*}(M)^{\Gamma} \right). $$
	
	By straightforward computation, we know that the idempotent $P_{t,\lambda}$ and $q$ have the following properties:
	\begin{enumerate}
        \item [(1)] 
        $\| P_{t,\lambda} \| \leq 64 $ ;
        \item [(2)]
        $ \mathrm{prop}(P_{t,\lambda}) \leq 6t^{-1} $ ;
		\item [(3)]
		$\left\| P_{t,\lambda} -q \right\| \leq 4 \operatorname{max} \left\lbrace \left\| 1-U_{t,\lambda}V_{t,\lambda} \right\| ,  \left\| 1-V_{t,\lambda}U_{t,\lambda} \right\| \right\rbrace  $ ;
		\item [(4)]
		$\operatorname{max} \left\lbrace \left\| 1-U_{t,\lambda}V_{t,\lambda} \right\| ,  \left\| 1-V_{t,\lambda}U_{t,\lambda} \right\| \right\rbrace \leq \left\| 1 - \chi(t^{-1}(d+d^*)+\lambda \hat{c}(\Xi)) ^2 \right\|  $ .
	\end{enumerate}

	\section{Equivariant Euler Operator And Localization Trick}

	Recall that $\Gamma$ denotes a countable discrete group. Let $M$ be an even-dimensional cocompact proper smooth $\Gamma$-manifold without boundary, equipped with a $\Gamma$-invariant Riemannian metric. Moreover, $\Xi$ is an equivariant vector field on $M$ that is non-degenerate and transverse to the zero-section.

    \subsection{Localization trick}
    
    Our first objective is to prove Theorem \ref{1.2}, specifically, to demonstrate that the K-homology class $[d+d^*]$ depends solely on the restriction of $d+d^*$ in the neighborhood of $\mathrm{Zero(\Xi)}$. We refer to this approach as the localization trick.
   
    Before proving Theorem \ref{1.2}, we first establish a lemma concerning the spectrum of $d+d^* + \lambda \hat{c}(\Xi)$. For convenience, we denote $d+d^{*}$ by $D$ and $t^{-1}(d+d^*) + \lambda \hat{c}(\Xi)$ by $D_{t,\lambda}$.
   
   \begin{lemma}
        \label{lemma1}
        For any positive number $\lambda$ and $t \geq 1$, we have, in the sense of the ordering of self-adjoint operators,
		$$ D_{t,\lambda}^2 \geq t^{-2}D^2+\lambda^2 \inf_{x\in M} \left\| \Xi(x) \right\| - Ct^{-1}\lambda ,   $$
		where $C>0$ is a positive constant depends only on $M$ and $\Xi$. 

        In particular, restricting to $ M \backslash \bigcup_{j\in \mathbb{N}} N_{\varepsilon}(x_{j}) $ \footnote{The terms \emph{restricting on $ M \backslash \bigcup_{j\in \mathbb{N}} N_{\varepsilon}(x_{j}) $} mean \emph{restricting $D$ on the space of the forms supported in $ M \backslash \bigcup_{j\in \mathbb{N}} N_{\varepsilon}(x_{j}) $}. }, there exists a positive constant $a_{\varepsilon}>0$ depends only on $M$, $\Xi$ and $\varepsilon$ such that 
		$$ D_{t,\lambda}^2 \geq a_{\varepsilon}\lambda^2-C \lambda .   $$
    \end{lemma}

    \begin{proof}
	Let $\omega \in \Omega(M) $ , say of $ \Omega^{even}(M) $. Then by straightforward calculation, for any $ x\in M $ we have that $D_{t,\lambda}^2 \omega (x) $ equals to 
	$$  t^{-2}D^2 \omega(x)+\lambda^2\left\|  \Xi(x) \right\|\omega (x)+ \lambda t^{-1}  D(\omega \cdot \Xi) (x) - \lambda t^{-1}  D(\omega) \cdot \Xi (x) ,$$
	where $\cdot$ denotes the Clifford multiplication.
	
	In terms of a local orthonormal frame $ e_1, \cdots, e_n $, we have 
	$$ D(\omega ) = \sum_{j} e_j \cdot \nabla_{e_j} \omega, $$
	and
	$$ D(\omega \cdot \Xi) = \sum_{j} e_j \cdot \nabla_{e_j} ( \omega \cdot \Xi ). $$
 
	Thus, we can obtain locally that
	\begin{align*}
		D_{t,\lambda}^2 \omega (x) = t^{-2}D^2 \omega (x)+\lambda^2\left\|  \Xi(x) \right\|\omega (x)+ \lambda t^{-1}  \left( \sum_{j} e_j \cdot \omega \cdot \nabla_{e_j} \Xi \right)  (x)  .  
	\end{align*}
	
	Denote $ T := D_{t,\lambda}^2-t^{-2}D^2 - \lambda^2 \hat{c}(\Xi), $
    and taking the inner product with $\omega$ in the equation above, we have:
	
	$$ \left| < T \omega , \omega >  \right| \leq  \left| \lambda t^{-1} < \sum_{j} e_j \cdot \omega \cdot \nabla_{e_j} \Xi ,\omega > \right| .   $$
	
	Since the $\Gamma$-manifold $M$ is proper and cocompact, we can view $M$ as a collection of copies of a compact part of $M$. Thus we obtain that $\nabla_{e_j} \Xi$ have a uniformly upper bound on $M$.
	
	It follows that
	$$ \left| < \sum_{j} e_j \cdot \omega \cdot \nabla_{e_j} \Xi ,\omega > \right|  \leq C \left\| \omega \right\|^2 , $$
	where $C$ is a positive constant depends only on the dimension of $M$ and upper bound of $\nabla_{e_j} \Xi$. 
	
	This means 
	$$ T \geq -Ct^{-1}\lambda . $$
	
	A similar argument applies if $\omega$ is a section of $\operatorname{Cliff}(TM)^{-}$. This completes the first part of the proof for Lemma~\ref{lemma1}. 
 
	Then, restricting to $ M \backslash \bigcup_{j\in \mathbb{N}} N_{\varepsilon}(x_{j}) $, we denote 
    $$a_{\varepsilon} :=\inf_{ x \in M \backslash \bigcup_{j\in \mathbb{N}} N_{\varepsilon}(x_{j})} \left\| \Xi(x) \right\|. $$  
    It is obvious that $a_{\varepsilon}>0$.
	\end{proof}
	
	We will now establish a lemma regarding the homotopy of equivariant idempotents in a Banach algebra.

    \begin{lemma}
    \label{lemma2}
        Let $A$ be a unital Banach algebra and $H$ a subgroup of the automorphism group of $A$. If $e$ and $f$ are $H$-invariant idempotents in $A$ and $ \| e - f \| < 1 / \| 2e-1 \| $ , then there is a $H$-invariant invertible element $v$ in $A$ with $ vev^{-1} =f $.
    \end{lemma}

    \begin{proof} 
        Set $v=1-e-f+2ef$ , then $v$ is $H$-invariant since $e$ and $f$ are $H$-invariant idempotents, and we can calculate that
    \begin{align*}
        ev&=e-e^2-ef+2e^2f=ef , \\
        vf&=f-ef-f^2+2ef^2=ef .
    \end{align*}
    Thus $ev=vf$. 
    
    We noticed that
    \begin{align*}
        1-v=e+f-2ef=2e^2-e+f-2ef=(2e-1)(e-f).
    \end{align*} 
    So when $ \| e - f \| < 1 / \| 2e-1 \| $ , we have $ \| 1-v \|<1 $ and thus $v$ is invertible. So $ev=vf$ leads to  $ vev^{-1} =f $.
    
    This completes the proof of Lemma \ref{lemma2}.
    \end{proof}

    Utilizing the two lemmas established previously, we will now commence the proof of Theorem \ref{1.2}.

    \begin{proof}[Proof of Theorem \ref{1.2}]
		First, we will use Lemma 3.1 to estimate the norm of $P_{t,\lambda}-q$. We decompose $ L^2\Omega^*(M) $ into the direct sum of the following three Hilbert spaces. 
    \begin{align*}
        H_0 :=& L^2 (\Omega ^* ( \bigcup_{j \in \mathbb{N}} N_{\varepsilon/4}(x_{j}) )) , \\
        H_1 :=& L^2( \Omega ^* ( \bigcup_{j \in \mathbb{N}} N_{\varepsilon/2}(x_{j})) \backslash \bigcup_{j \in \mathbb{N}} N_{\varepsilon/4}(x_{j})  ), \\
        H_2 :=& L^2( \Omega ^* ( M \backslash \bigcup_{j \in \mathbb{N}} N_{\varepsilon/2}(x_{j}) )). 
    \end{align*}

	Restricting to $ H_1 \oplus H_2 $ , we have 
	$$ D_{t,\lambda} ^2 \geq \lambda^2-C\lambda . $$
	Thus, for any positive number $\delta$, we choose $\lambda$ sufficiently large to make 
    $$ D_{t,\lambda} ^2 \geq  256\delta^{-2} . $$
	
	Then we could estimate the spectrum of $\chi(D_{t,\lambda})$ restricting on $ H_1 \oplus H_2 $. To do this, we need a lemma that addresses the relationship between the functional calculus of an operator and the functional calculus of its restriction, which follows essentially from \cite[Lemma 2.5]{restr}. This lemma states that we can estimate the spectrum of $\chi(D_{t,\lambda})|_{H_1 \oplus H_2}$ by considering the spectrum of $\chi(D_{t,\lambda}|_{H_1 \oplus H_2})$ since the support of the Fourier transformation of $\chi$ is compact.

    Since we choose a specific $\chi$ as
    $$ \chi(x) = \frac{2}{\pi} \int_{0}^{x} \frac{1-\operatorname{cos}y}{y^2} dy, $$ 
    which satisfies that $ 1-\chi(x)^2 <\frac{8}{\pi \left| x \right| } $, we have
	$$ 1- \chi(D_{t,\lambda})^2 < \frac{8}{\pi} \cdot \frac{\delta}{16} < \frac{\delta }{4}, $$
    restricting to $H_{1}\oplus H_{2}$.

	Recall the construction of the local index of $D$. We have that $\left\| P_{t,\lambda} -q \right\| < \delta  $ restricting on $ H_1 \oplus H_2 $, therefore $P_{t,\lambda} -q$ can be perturbed by an operator with norm less than $\delta$ to be an equivariant idempotent which is supported in $ H_0 \oplus H_1 $ by the following step.
	
	We write $P_{t,\lambda}-q$ as a $3\times3$ matrix on $H_0 \oplus H_1 \oplus H_2$ as 
    \begin{align*}
        \begin{bmatrix}
            p_{11}(t) & p_{12}(t) & p_{13}(t) \\
            p_{21}(t) & p_{22}(t) & p_{23}(t) \\
            p_{31}(t) & p_{32}(t) & p_{33}(t)
        \end{bmatrix}
        -
        \begin{bmatrix}
            q_{11} & q_{12} & q_{13} \\
            q_{21} & q_{22} & q_{23} \\
            q_{31} & q_{32} & q_{33}
        \end{bmatrix}
    \end{align*}
	for some fixed $\lambda$. Since the propogation speed of $q$ is 0 and of the propogation speed $P_{t,\lambda}$ converges to 0 as $t\rightarrow \infty$, so $q_{ij}=0$ for any $i\not=j$ and we can choose $t$ large enough to make $p_{13}(t)=p_{31}(t)=0$. Thus, the matrix below can be simplified as
    \begin{align*}
        \begin{bmatrix}
            p_{11}(t) & p_{12}(t) & 0 \\
            p_{21}(t) & p_{22}(t) & p_{23}(t) \\
            0 & p_{32}(t) & p_{33}(t)
        \end{bmatrix}
        -
        \begin{bmatrix}
            q_{11} & 0 & 0 \\
            0 & q_{22} & 0 \\
            0 & 0 & q_{33}
        \end{bmatrix}
    \end{align*}

    Since $\left\| P_{t,\lambda} -q \right\| < \delta  $ restricting on $ H_1 \oplus H_2 $ , $ \| p_{23}(t) \| , \| p_{32}(t) \| \leq \delta $ . Then, the norm of 
    \begin{align*}    
        \begin{bmatrix}
            p_{11}(t) & p_{12}(t) & 0 \\
            p_{21}(t) & p_{22}(t) & p_{23}(t) \\
            0 & p_{32}(t) & p_{33}(t)
        \end{bmatrix}
        -
        \begin{bmatrix}
            p_{11}(t) & p_{12}(t) & 0 \\
            p_{21}(t) & p_{22}(t) & 0 \\
            0 & 0 & p_{33}(t)
        \end{bmatrix}
    \end{align*}
    is no more than $2\delta$.

    Therefore, we have 
    \begin{align*}
        \frac{1}{\|2P_{t,\lambda}-1\|} \geq \frac{1}{2\|P_{t,\lambda}\|+1} \geq \frac{1}{2\times64+1} = \frac{1}{129} .
    \end{align*}

    So we can choose $\delta < 1 / 258 $ to make the $2\delta < 1 / \|2P_{t,\lambda}-1\|$. Then by Lemma3.2 , $P_{t,\lambda}-q$ is equivariantly homotopic to 
    \begin{align*}
        \begin{bmatrix}
            p_{11}(t) & p_{12}(t) & 0 \\
            p_{21}(t) & p_{22}(t) & 0 \\
            0 & 0 & p_{33}(t)
        \end{bmatrix}
        -
        \begin{bmatrix}
            q_{11} & 0 & 0 \\
            0 & q_{22} & 0 \\
            0 & 0 & q_{33}
        \end{bmatrix}.
    \end{align*}

    We denote
    \begin{align*}
        P|_{ H_{0}\oplus H_{1} }
        :=
        \begin{bmatrix}
            p_{11}(t) & p_{12}(t) & 0 \\
            p_{21}(t) & p_{22}(t) & 0 \\
            0 & 0 & 0
        \end{bmatrix}
        -
        \begin{bmatrix}
            q_{11} & 0 & 0 \\
            0 & q_{22} & 0 \\
            0 & 0 & 0
        \end{bmatrix}
    \end{align*}
    and
    \begin{align*}
        P|_{ H_{2} }
        :=
        \begin{bmatrix}
            0 & 0 & 0 \\
            0 & 0 & 0 \\
            0 & 0 & p_{33}(t)
        \end{bmatrix}
        -
        \begin{bmatrix}
            0 & 0 & 0 \\
            0 & 0 & 0 \\
            0 & 0 & q_{33}
        \end{bmatrix}.
    \end{align*}
    
    We noticed that the support of $ P|_{ H_{0}\oplus H_{1} } $
    lies in the $ \bigcup_{j \in \mathbb{N}} N_{\varepsilon/2}(x_{j}) $ and thus lies in a compact subset of $ \bigcup_{j \in \mathbb{N}} N_{\varepsilon}(x_{j}) $. 
    
    Therefore, the K-theory class of $ P|_{ H_{0}\oplus H_{1} } $
    lies in the image of 
	$$ i_{*} \left( K_{0}\left( RC_{L}^{*} (\bigcup_{j \in \mathbb{N}} N_{\varepsilon}(x_{j})) ^{\Gamma}\right)  \right) \subseteq K_{0}\left( C_{L}^{*}(M) ^{\Gamma}\right) .  $$

    Also since $ \| p_{33}(t) - q_{33} \| \leq \delta $, we could choose $\delta$ small enough to make $ \| p_{33}(t) - q_{33} \| \leq 1 / \| 2q_{33}-1 \| $ , which means that the K-theory class of $ P|_{H_{2}} $ is $0$.
    
    This completes the proof of Theorem \ref{1.2}. 
	\end{proof}

    \begin{definition}
    \label{def}
        Using the isomorphism
        $$ 
        K_{0}\left( RC_{L}^{*} (\bigcup N_{\varepsilon} ( x_{j} ))^{\Gamma} \right) 
        \cong 
        \bigoplus_{i=1}^{m} K_{0}\left( RC_{L}^{*} ( N_{\varepsilon} ( x_{j} ))^{\Gamma_{j}} \right) , $$
        we decompose the pre-image of the K-theory class of $ P|_{ H_{0}\oplus H_{1} } $ under $i_{*}$ as
        $$ \sum_{j=1}^{m} [d+d^{*}]|_{N_{\varepsilon}(x_{j})} \in \bigoplus_{i=1}^{m} K_{0}\left( RC_{L}^{*} ( N_{\varepsilon} ( x_{j} ))^{\Gamma_{j}} \right) , $$
        where $ [d+d^{*}]|_{N_{\varepsilon}(x_{j})} \in K_{0}\left( RC_{L}^{*} ( N_{\varepsilon} ( x_{j} ))^{\Gamma_{j}} \right) $ for each $j = 1, \cdots , m $.

    \end{definition}

    \textbf{Remark:} The class \( [d+d^{*}]|_{N_{\varepsilon}(x_{j})} \in K_{0}\left( RC_{L}^{*} ( N_{\varepsilon} ( x_{j} ))^{\Gamma_{j}} \right) \) is actually determined by the operator \( t^{-1}(d+d^{*}) + \lambda \hat{c}(\Xi) \) restricted to \( H_{0}\oplus H_{1} = L^{2} \Omega^{*} ( \bigcup_{j \in \mathbb{N}} N_{\varepsilon/2}(x_{j})) \). However, to emphasize the operator \( d+d^{*} \) on which we focus and the neighborhood \( N_{\varepsilon}(x_{j}) \) where we localized our question, we use the notation \( [d+d^{*}]|_{N_{\varepsilon}(x_{j})} \).

    \subsection{ Local information }

    Now, it has been proven that the K-homology class \( [d+d^*] \) depends only on the restriction of \( d+d^* \) to the neighborhood \( \bigcup_{j \in \mathbb{N}} N_{\varepsilon}(x_{j}) \) of the zero set of \( \Xi \). Next, we will examine how much information the map \( N_{\varepsilon} (x_{j}) \hookrightarrow M \) conveys for each \( j = 1, \cdots, m \).
    
    In this section, our primary goal is to prove Theorem \ref{1.3}, which requires us to consider the relationship between the class of the Bott operator in $ K_{0}^{\Gamma_{j}}(\mathbb{R}^n) $ and $ \mathrm{ind} (\Xi,x_{j}) \cdot [d+d^{*}]|_{N_{\varepsilon}(x_{j})} $.

    First, we should choose a specific local coordinate $ ( y_1 , \cdots , y_n ) $ on $ N_{\varepsilon} ( x_{j} ) $ to identify $ N_{\varepsilon} ( x_{j} ) $ as $ \mathbb{R}^n $ with a proper $\Gamma_{j}$ action. And we can write $\Xi$ as
    $$ - y_{1} \frac{ \partial }{ \partial y_{1} } - \cdots - y_{r} \frac{ \partial }{ \partial y_{r} } + y_{r+1} \frac{ \partial }{ \partial y_{r+1} } + \cdots + y_{n} \frac{ \partial }{ \partial y_{n} },
    \footnote{ Here we view  $\Xi$ as the gradient of a Morse function, more details can be found in \cite{wass}.  } $$
    for some $ r \in \{ 1 , \cdots , n \} $. So we can locally view $\Xi$ as the vector field above on $ \mathbb{R}^n $.

    The Bott operator on $\mathbb{R}^n$, denoted by $C$, is the operator of right Clifford multiplication by 
    $$ y_{1} \frac{ \partial }{ \partial y_{1} } + \cdots + y_{r} \frac{ \partial }{ \partial y_{r} } + y_{r+1} \frac{ \partial }{ \partial y_{r+1} } + \cdots + y_{n} \frac{ \partial }{ \partial y_{n} },
    $$
    twisted by the grading.

    \begin{lemma}
    \label{H}
        The Bott operator on $\mathbb{R}^n$ determines a class $ [C] $ in the K-theory group $ K_{\Gamma_{j}}^{0} ( \mathbb{R}^n ) $, we have
        $$ [ \hat{c} ( \Xi ) ] = (-1)^{r} \cdot [C] = \mathrm{ind} (\Xi, x_{j}) \cdot [C]. $$
    \end{lemma}

    This lemma can be proved by considering only the cases of $r=0$ and $r=1$; when $r=0$ this lemma is obvious, and when $r=1$ one can verify that $ [\hat{c}(\Xi) \oplus C]=0 $ directly by definition.
    
    Then, we can prove Theorem \ref{1.3}.

    \begin{proof}[Proof of Theorem \ref{1.3}]

       Let $ D_{AS} $ be the Aityah-Singer operator on $\mathbb{R}^n$. Like the non-equivariant case in \cite{yu}, the equivariant Poincar\'e duality on $ \mathbb{R}^n $ is just 
        $$ \otimes [ D_{AS} ]: K^{0}_{\Gamma_{j}} (\mathbb{R}^n)
        \rightarrow
        K_{0} ( RC_{L}^{*}(\mathbb{R}^n)^{\Gamma_{j}} ) ,
        [\alpha] \mapsto [ \alpha \otimes  D_{AS} ] . $$
        So the class of the $ \hat{c}(\Xi) $ in $ K^{0}_{\Gamma_{j}} (\mathbb{R}^n) $ can be identified as $$ [ \hat{c} (\Xi) ]\otimes [D_{AS}  ] \in K_{0} ( RC_{L}^{*}(\mathbb{R}^n)^{\Gamma_{j}} ).$$

        Next, we regard $ K_{0} ( RC_{L}^{*}( N_{\varepsilon}( x_{j} ) )^{\Gamma_{j}} ) $ as $ K_{0} ( RC_{L}^{*}(\mathbb{R}^n)^{\Gamma_{j}} ) $, and observe that the K-theory class
        $ [\hat{c} (\Xi)] \otimes [D_{AS}] $
        is the same as 
        $ [ \hat{c}(\Xi) \otimes 1 + 1 \otimes D_{AS} ]$ (\cite[Chapter 9]{yu} ) which is exactly $ [d+d^{*}+\hat{c}(\Xi) ].$

        Hence the discussion before leads to an isomorphism
        $$ \mathrm{PD}_{j} : K_{0}^{\Gamma_{j}}(\mathbb{R}^n) \stackrel{\cong}{\longrightarrow} K_{0}\left( RC_{L}^{*} (N_{\varepsilon}(x_{j}))^{\Gamma_{j}} \right), $$
        with $ \mathrm{PD}_{j} ( [C] ) = \mathrm{ind} (\Xi,x_{j}) \cdot [d+d^{*}]|_{N_{\varepsilon}(x_{j})} $.

    This completes the proof of Theorem \ref{1.3}.

    \end{proof}

 \section{Higher index map and Poincar\'e-Hopf Theorem on K-homology level}

    \subsection{Non-equivariant case}

    In this section, we begin by addressing a more straightforward, yet still significant, case: the non-equivariant case. By assuming the group $\Gamma$ is trivial, we reduce the problem discussed earlier to its non-equivariant form. We start by considering Bott periodicity. Since the dimension $n$ of $M$ is even, Bott periodicity provides the following result:

    \begin{lemma}[Bott periodicity]
    \label{BP_nonequivariant}
        There exists an isomorphism
        $$ \beta_{j} : \mathbb{Z} \cong K^0 ( \mathrm{pt} ) \xrightarrow[\text{Bott periodicity}]{\cong} K^{0}(\mathbb{R}^n), $$
        and this isomorphism maps the class of the trivial bundle over $ \mathrm{pt} $ in $K^0(\mathrm{pt})$ to the class of the Bott operator $ \hat{c}(\Xi) $ in $ K^{0} ( \mathbb{R}^{n}) $.
    \end{lemma}

    This leads us to the following result:

    \begin{theorem}
    \label{nonequivariant}
        The K-homology class of the Euler characteristic operator $d + d^{*}$ can be understood in two distinct ways:
        \begin{enumerate}
            \item The K-homology class of $d + d^{*}$ lies in the image of $\chi(M) \in \mathbb{Z}$ under a group homomorphism $ \mathbb{Z} \rightarrow K_{0}(M) $, which will be constructed in the proof.
            \item The index of the Euler characteristic operator, i.e., the Euler characteristic of $M$, is given by the sum of the indices of the zero points of $\Xi$:
            $$ \chi(M) = \sum_{j=1}^{m} \mathrm{ind} (\Xi, x_{j}). $$
            \end{enumerate}
    \end{theorem}

    We should note that part one of Theorem \ref{nonequivariant} is analogous to the main result in \cite{rj}, while part two of the theorem is the classical Poincar\'e-Hopf formula.

    \begin{proof}
    
        Consider the maps
        $$ \bigcup_{j=1}^{m} N_{\varepsilon}(x_{j}) \xrightarrow{i} M \xrightarrow{c} \mathrm{pt}, $$
        where $i$ denotes the inclusion of $ N_{\varepsilon}(x_{j}) $ into $M$, and $c$ represents the map collapsing $M$ to a single point.

        These maps induce the following K-homology maps:
        $$ K_{0}(\bigcup_{j=1}^{m} N_{\varepsilon}(x_{j})) \xrightarrow{i_{*}} K_{0}(M) \xrightarrow{c_{*}} K_{0}(\mathrm{pt}). $$

        Note that the map $ c_{*} $ corresponds to the index map, so we have $ c_{*}([d+d^{*}]) = \mathrm{ind}(d+d^*) = \chi(M) $. By Theorem \ref{1.2}, $ [d+d^{*}] \in K_{0}(M) $ lies in the image of $i_*$.

        Since $ c_{*} \circ i_{*} $ defines a surjection from $ K_{0}(\bigcup_{j=1}^{m} N_{\varepsilon}(x_{j})) \cong \bigoplus_{j=1}^{m} \mathbb{Z} $ to $ K_{0}(\mathrm{pt}) \cong \mathbb{Z} $, there exists an element $ \alpha \in K_{0}(\bigcup_{j=1}^{m} N_{\varepsilon}(x_{j})) $ such that $ c_{*} \circ i_{*} (\alpha) $ generates $ K_{0}(\mathrm{pt}) $, and $ \chi(M) \cdot i_{*} (\alpha) = [d+d^*] $.

        Consider the subgroup $ \langle \alpha \rangle $ of $ K_{0}(\bigcup_{j=1}^{m} N_{\varepsilon}(x_{j})) $ generated by $ \alpha $. As $ K_{0}(\bigcup_{j=1}^{m} N_{\varepsilon}(x_{j})) $ is torsion-free, we have $ \langle \alpha \rangle \cong \mathbb{Z} $, and under the map $ i_{*} : \langle \alpha \rangle \cong \mathbb{Z} \rightarrow K_{0}(M) $, we find that $ \chi(M) $ maps to $ [d+d^*] $.

        This establishes part one of Theorem \ref{nonequivariant}.

        To prove part two of Theorem \ref{nonequivariant}, observe that 
        $$ \chi(M) = c_{*} \left( \sum_{j=1}^{m} \mathrm{ind} (\Xi,x_{j}) \cdot i_{*} \circ r_{*} \circ \varphi_{j} (1) \right) \in \mathbb{Z}. $$

        Thus, the proof is complete.
    \end{proof}

    \subsection{equivariant case}

    In this section, we discuss our main theorem, which follows directly from the proofs of Theorems \ref{1.2} and \ref{1.3}. Our main result, Theorem \ref{main}, is a corollary of these two theorems and relies on Bott periodicity.

    First, we state the equivariant Bott periodicity theorem:

    \begin{lemma}[Equivariant Bott periodicity]
        \label{BP_equivariant}
        There exists an isomorphism
        $$ \beta_{j} : K_{0}(C^{*}(\Gamma_{j})) \cong K_{\Gamma_{j}}^0(\mathrm{pt}) \xrightarrow[\text{Bott periodicity}]{\cong} K_{\Gamma_{j}}^0(\mathbb{R}^n), $$
        which maps the class of the trivial representation $I_{j}$ of $\Gamma_{j}$ in $K_{0}(C^{*}(\Gamma_{j}))$ to the class of the Bott operator $\hat{c}(\Xi)$ in $K_{\Gamma_{j}}^0(\mathbb{R}^n)$.
    \end{lemma}

    The proof of Lemma \ref{BP_equivariant}, which establishes the Bott periodicity of $\mathbb{R}^n$ (with even dimension) at the level of equivariant K-theory, can be found in \cite{ph}.

    \begin{proof}[Proof of Theorem \ref{main}]
        
        To prove Theorem \ref{main}, we first need to construct the map $ \varphi_{j} $.

        By Theorem \ref{1.3} and Bott periodicity, we define 
        $ \varphi_{j} = \mathrm{PD}_{j} \circ \beta_{j} $.
        $$ 
        \begin{tikzcd}
                                                                                                                                                     & K^{0}_{\Gamma_{j}} ( \mathbb{R}^{n}) \arrow[rd, "\text{Poincar\'e duality}"] &                                                 \\
K_{0}(C^*(\Gamma_{j}))          \cong          K^{0}_{\Gamma_{j}} ( \mathrm{pt} )  \arrow[ru, "\text{Bott periodicity}"] \arrow[rr, "\varphi_{j}", dashed] &                                                                              & K_{0} ( RC_{L}^{*}(\mathbb{R}^n)^{\Gamma_{j}} )
\end{tikzcd} .
$$

     Under the isomorphism, we have 
     $$ \varphi_{j} ( [I_{j}] ) = \mathrm{ind} ( \Xi , x_{j} ) [ d+d^* ]|_{N_{\varepsilon}(x_{j})}. $$

    Next, consider the induction:
    $$ 
    \bigoplus_{j=1}^{m} K_{0}(RC_{L}^*( N_{\varepsilon}(x_{j}) )^{\Gamma_{j}})
    \xrightarrow[r_{*}]{\cong}
    K_{0}\left( RC_{L}^{*} (\bigcup_{j \in \mathbb{N}} N_{\varepsilon}(x_{j})) ^{\Gamma} \right) .
    $$
    The detailed properties of induction can be found in Section~6.5 of \cite{yu}.

    Finally, Theorem \ref{1.2} induces the map

    $$ 
    K_{0}\left( RC_{L}^{*} (\bigcup_{j \in \mathbb{N}} N_{\varepsilon}(x_{j})) ^{\Gamma} \right)
    \xrightarrow[]{i_{*}}
    K_{0}^{\Gamma} (M),
    $$
    and we have
    $$[d+d^*]=\sum_{j=1}^{m} \mathrm{ind} (\Xi,x_{j})\cdot i_{*} \circ r_{*} \circ \varphi_{j}  ([I_{j}]).$$

    This completes the proof of Theorem \ref{main}.
    
    \end{proof}

	\bibliography{main}{}
\bibliographystyle{amsplain}

\end{document}